\documentclass[12pt,a4paper]{amsart}
\usepackage[colorlinks=true,linkcolor=blue,urlcolor=blue,citecolor=blue]{hyperref}
\usepackage{mlmodern}
\usepackage[dvipsnames]{xcolor} 
\usepackage{amsmath,color}
\usepackage{amsthm, enumerate}
\usepackage{amssymb}
\usepackage[all]{xy}
\usepackage{epsfig, soul}

\usepackage{geometry}

\newtheorem{theorem}{Theorem}[section]
\newtheorem{proposition}[theorem]{Proposition}
\newtheorem{corollary}[theorem]{Corollary}
\newtheorem{lemma}[theorem]{Lemma}
\theoremstyle{definition}

\newtheorem{example}[theorem]{Example}
\newtheorem{remark}[theorem]{Remark}
\newtheorem{questions}[theorem]{Question}

\def\ep{\varepsilon}
\def\E{\mathcal E}
\def\dsum{\displaystyle \sum}
\newcommand{\bff}{{\boldsymbol{f}}}
\newcommand{\C}{\mathbb{C}}

\newcommand{\Z}{\mathbb{Z}}
\def\dimsup{\overline{{\dim}}}
\def\diminf{\underline{{\dim}}}
\def\Span{{\rm span}}
\def\coe{{\rm coe \,}}

\author[D. Carando]{Daniel Carando}

\address{Departamento de Matem{\'a}tica\\ Facultad de Cs. Exactas y Naturales\\
		Universidad de Buenos Aires and IMAS-UBA-CONICET\\ Int.~G{\"u}iraldes s/n, 1428\\ Buenos Aires, Argentina.}
\email{dcarando@dm.uba.ar}

\author[C. D'Andrea]{Carlos D'Andrea}

\address{Departament de Matem\`atiques i Inform\`atica, Universitat de Barcelona. Gran
Via 585, 08007 Barcelona, Spain \&  Centre de Recerca Matem\`atica, Edifici C, Campus Bellaterra, 08193 Bellaterra, Spain}
\email{cdandrea@ub.edu}

\author[L. A. Torres ]{Leodan A. Torres}

\address{IMAS - UBA - CONICET \\
Pabell\'on I\\ Facultad de Cs. Exactas y Naturales \\ Universidad de Buenos Aires \\
(1428) Buenos Aires\\ Argentina.}

\email{ltorres@dm.uba.ar}

\author[P. Turco]{Pablo Turco}

\address{IMAS - UBA - CONICET \\
Pabell\'on I\\ Facultad de Cs. Exactas y Naturales \\ Universidad de Buenos Aires \\
(1428) Buenos Aires\\ Argentina.}
\email{paturco@dm.uba.ar}

\thanks{The first author was partially supported by CONICET PIP  11220200102366CO and
ANPCyT PICT 2018-04104.The second author was partially supported by the Spanish MINECO research project PID2019-104047GB-I00.
The third author has a posdoc fellowship from CONICET. The fourth author was partially supported by ANPCyT PICT 2019-00080 and CONICET PIP 112202001101609CO}

\subjclass[2020]{Primary 46G20, 46B06, 47B06; Secondary 46B50, 28A75}

\keywords{Entropy numbers, Box dimension, Banach spaces, Polynomials and Holomorphic functions.}

\begin{document}

\title[Entropy numbers and box dimension]{Entropy numbers and box dimension of polynomials and holomorphic functions}

\begin{abstract}
We study  entropy numbers and  box dimension of (the image of) homogeneous polynomials and holomorphic functions between Banach spaces. First, we see that entropy numbers and box dimensions of subsets of Banach spaces are related. We show that the box dimension of the image of a ball under a homogeneous polynomial is finite if and only if it spans a  finite-dimensional subspace, but this is not true for holomorphic functions. Furthermore, we relate the entropy numbers of a holomorphic function to those of the polynomials of its Taylor series expansion. As a consequence, if the box dimension of the image of a ball by a holomorphic function $f$ is finite, then the entropy numbers of the polynomials in the Taylor series expansion of $f$ at any point of the ball belong to $\ell_p$ for every $p>1$.
\end{abstract}

\maketitle

\section*{Introduction}

In this note, we study the {\it compactness} of the image of homogeneous polynomials and holomorphic functions between Banach spaces. Let $E$ and $F$ be Banach spaces, $U\subset E$ an open set and $f\colon U\rightarrow F$ a holomorphic mapping.
Whenever $f$ maps a ball $B\subset U$ onto a relatively compact set, we use entropy numbers or box dimension (see the definitions in Section \ref{Sec: Dimension y entropy numbers}) to \emph{measure} the compactness of $f(B)$.
For $x_0\in U$, let $P_mf(x_0)$ be the $m$-homogeneous polynomial of the Taylor series expansion of $f$ at $x_0$.
This article was originally motivated by the following question, posed by Richard Aron to the fourth author: given $\ep>0$ and some ball $B\subset U$, can we relate the \emph{degree of compactness} (in terms of entropy numbers or box dimension) of  $f(B)$ and that of $P_mf(x_0)(B)$? Similar  questions were addressed in \cite{AMR, ArSc, LaTur,Tur} for  other ways of measuring the compactness of a set. Also, entropy numbers and, in general, the theory of $s$-numbers and quasi $s$-numbers of multilinear operators was treated in \cite{CaRu, FMS, FMS2, FerSi}.

Before dealing with this and other similar  questions, we will see in Section \ref{Sec: Fractal dimension} that entropy numbers and box dimension are closely related. Indeed, Proposition \ref{prop: dimension} essentially states that, for a connected set $K$, the box dimension of $K$ is finite {if and only if} the entropy numbers of $K$ decay exponentially.

Also, we show an example of a connected set $K$ such that its box dimension is infinite and the sequence $(e_n(K))_n \in \ell_1$.
In Section~\ref{Sec: Dimension y entropy numbers} we study $m$-homogeneous polynomials and holomorphic functions. In particular, we relate the box dimension of the image and the linear dimension of the subspace it spans. For example, we present a function $f$ defined on the complex unit disk $\Delta$ such that, for every smaller disk $D\subset \Delta$,  the (upper) box dimension of $f(D)$ is 2  while $f(D)$ spans an infinite dimensional subspace (Examples~\ref{Example: holomorphic function} and~\ref{Example: extension}). On the other hand, we see that such an example cannot exist for $m$-homogeneous polynomials. In fact, if the image of an $m$-homogeneous polynomial spans an infinite dimensional subspace, then its box dimension must be infinite (Theorem~\ref{thm: BoxdimensionPoly}).

Finally, in Section~\ref{Sec: Holo vs pol}, for a given holomorphic function $f\colon U\rightarrow F$, $x_0\in U$ and $\ep>0$ we obtain in Lemma~\ref{Lemma: p.prin} a relationship between the entropy numbers $e_n(f(x_0+\ep B_E))$ and $e_N(P_mf(x_0)(B_E))$, where $n$ and $N$ are related (here, $B_E$ is the unit ball of $E$). As a consequence, we see in Proposition~\ref{prop: entropy en ellp} that if the (upper) box dimension of $f(x_0+\ep B_E)$ is finite, then the sequence $(e_n(P_mf(x_0)(B_E)))_n $ belongs to $ \ell_p$ for every $p>1$.
The problem of measuring $P_mf(B_E)$ (in some sense) in terms of the image of $f$ is closely related to the problem of measuring  the absolutely convex  hull of a set $K$ in terms of $K$ (see the proof of \cite[Proposition~3.4]{ArSc}). For entropy numbers, this geometric problem was studied by many authors (see, for example,  \cite{CHR, CKP, KlRu, FerSi}  and the references therein). Our Proposition~\ref{prop: entropy en ellp} seems to provide sharper results than those obtained from the proof of \cite[Proposition~3.4]{ArSc} together with the known relationships between the entropy numbers of a set and its absolutely convex hull (see the discussion after  Proposition~\ref{prop: entropy en ellp}).

\medskip 

Throughout the article, we  consider complex Banach spaces $E$ and $F$ and write $B_E$ for the open unit ball of $E$ and $\Delta$ for the open unit disc in $\mathbb C$. Also, $\ell_p^N$ stands for $\mathbb C^N$ endowed with the $\ell_p$ norm.  We write $\epsilon_j$ for  the vector/sequence which has 1 in the $j$-th coordinate and 0 elsewhere. This notation will be used both in $\mathbb C^N$ and in $\ell_p$.  

For a $m$-homogenoeus polynomial $P$, by $\overset \vee P(x_1,\ldots, x_m)$ we denote the unique symmetric $m$-linear operator such that $P(x)=\overset \vee P(x,\ldots, x)=\overset \vee P(x^m)$.


\section{Entropy numbers, covering numbers and box dimension} \label{Sec: Fractal dimension}

A possible way to refine the concept of compactness in a Banach space is via the so-called entropy numbers. Recall that the \emph{$n$-th entropy number}  $\E_n(K)$ of a set $K$ of a metric space $(X,d)$  is defined as
$$
\E_n(K)\colon=\inf\left\{\ep>0 \colon \exists \ x_1,\ldots,x_n \in X \colon K \subset \bigcup_{i=1}^{n} B_X(x_i,\ep)\right\},
$$
where $B_X(x,\ep)=\{\widetilde x \in X\colon d(x,\widetilde x)<\ep\}$.  The \emph{$n$-th dyadic entropy numbers} of $K$ are given by $e_n(K)=\E_{2^{n-1}}(K)$. A subset $K$ of a Banach space is relatively compact if and only if the sequence $(e_n(K))_{n\in \mathbb N} \in c_0$. Stronger conditions on the decay rate of $(e_n(K))_{n\in \mathbb N} \in c_0$ lead to stronger versions of compactness.

The concept of covering numbers is closely related to that of entropy numbers. For a bounded set $L\subset X$ and $\ep>0$, the \emph{$\ep$-covering number} $N(L,\ep)$ is  given by
$$
N(L,\ep)\colon =\min\left\{ n \in \mathbb N\colon \exists \ x_1,\ldots,x_n \in X \colon L\subset \bigcup_{i=1}^{n} B_X(x_i,\ep)\right\}.
$$
Also, the \emph{upper} and \emph{lower box counting dimension of $L$} are given by
$$
\dimsup_{B} L=\limsup_{\ep\to 0^+} \dfrac{\log N(L,\ep)}{-\log(\ep)} \quad \text{and} \quad
\diminf_{B}L= \liminf_{\ep\to 0^+} \dfrac{\log N(L,\ep)}{-\log(\ep)}.
$$
In the case that $\lim\limits_{\ep\to 0^+} \dfrac{\log N(L,\ep)}{-\log(\ep)}$ exists, we say that the \emph{box counting dimension} of $L$  is
$$
\dim_B L=\lim_{\ep\rightarrow 0^+} \dfrac{\log N(L,\ep)}{-\log(\ep)}.
$$
We remark that, if $L\subset X$ is not totally bounded, then $\diminf_B L=\infty$.
Also, if  $g\colon X\rightarrow Y$ is a Lipschitz function and $L\subset X$ is a bounded subset, then $\dimsup_B g(L)\leq \dimsup_B L$ and $\diminf_B g(L)\leq \diminf_B L$.
We refer to \cite{Fal} for the basics of the theory of fractal geometry.

The (linear) dimension of a subset  $U$  of a complex vector space $\mathbb V$ is the dimension of $\Span \{U\}$ and is denoted by $\dim U$.  Note that for an open and bounded set $U\subset \mathbb C$ we have  $\diminf_B U=\dimsup_B U=2$, while $\dim \ U=1$.

The following result can be found in \cite[Chapter~3.1]{Fal} for sets in $\mathbb R^N$, but the proof works line by line for general metric spaces?

\begin{proposition}\label{prop:deltas} Let $(X,d)$ a metric space and $L\subset X$ a bounded set. If $(\delta_n)_n$ is a {decreasing} sequence of real numbers such that $\lim\limits_{n\to \infty} \delta_n=0$ and $\liminf\limits_{n\to \infty} \dfrac{\delta_{n+1}}{\delta_n}>0$, then
$$
\dimsup_{B} L=\limsup_{n\to \infty} \dfrac{\log N(L,\delta_n)}{-\log(\delta_n)} \quad \quad
\diminf_{B}L= \liminf_{n\to \infty} \dfrac{\log N(L,\delta_n)}{-\log(\delta_n)}.
$$
\end{proposition}

The following proposition gives the connection between upper box dimension and the asymptotic behaviour of the entropy numbers.

\begin{proposition}\label{prop: dimension} Let $(X,d)$ be a metric space and $L\subset X$ a connected and totally bounded set. Then, $\dimsup_B L<\infty$ if and only if $\limsup\limits_{n\to \infty} e_n(L)^{1/n}<1$.
\end{proposition}
\begin{proof}
Suppose that $\dimsup_{B}=d<\infty$. By \cite[Corollary~5]{Foll} (which is stated for subsets of $\mathbb R^N$ but holds in general metric spaces), there exists $n_\ep \in \mathbb N$ such that for every $n>n_\ep$ we have $$e_n(L)\leq 2(2^{n-1}+1)^{{-\frac{1}{d+\ep}}}\le 2^{1-\frac{n}{d+\ep}}.$$
This gives

$\limsup\limits_{n\to \infty} e_n(L)^{1/n}\leq 2^{-\frac{1}{d+\ep}}<1$.
	
Conversely, suppose that $\limsup\limits_{n\to \infty} e_n(L)^{1/n}<1$ and take $0<\beta<1$,  $n_0\in  \mathbb N$ such that $e_n(L)< \beta^n$ for all $n\ge n_0$.
	For each $n$, the definiton of  $e_n(L)$ gives us $x_1,x_2,\ldots,x_{2^n-1}\in X$ such that $L\subset\bigcup_{j=1}^{2^{n-1}}B(x_j,2e_n(L))$. Therefore, $N(L,2e_n(L))\leq 2^{n-1}$ and for $n\ge n_0$ we have
		\begin{align*}
		  \frac{\log\left(N(L,2e_n(L))\right)}{-\log(2e_n(L))} & \leq \frac{(n-1)\log 2}{-\log(2e_n(L))}<\frac{(n-1)\log 2}{-\log(2\beta^n)} \\
		   & = \frac{(n-1) \log 2}     {-\log(2) - n \log(\beta)} \underset{n \to \infty}{\longrightarrow} -\frac{\log(2)}{\log(\beta)} <+\infty.
		\end{align*}
In order to conclude that $\dimsup_B L<\infty$, we want to use Proposition \ref{prop:deltas} with $\delta_n=2e_n(L)$, which are clearly decreasing and convergent to zero (since $L$ is totally bounded). The fact that $\liminf\limits_{n\to \infty} \dfrac{\delta_{n+1}}{\delta_n}>0$  follows from  Lemma \ref{prop: liminf en} below.
\end{proof}

\begin{lemma}\label{prop: liminf en} Let $(X,d)$ be a metric space and $L\subset X$ a connected set. Then, $$\liminf_{n\to \infty} \dfrac{e_{n+1}(L)}{e_n(L)}\geq \dfrac 15.$$
\end{lemma}
\begin{proof}

Fix $n \in \mathbb N$, take $\delta>0$  and set $r=(1+\delta) e_{n+1}(L)$. There exists a subset $L_0=\{x_1,x_2,\ldots,x_{2^{n}}\}$ of $X$ such that $L\subset \bigcup_{x\in L_0} B_X(x,r)$.

First, we claim that for $x \in L_0$, there exists $\widetilde x \in L_0$, $x\neq \widetilde x$ such that $d(x,\widetilde x)< 2r$. Indeed, suppose that $B_X(x,r)\cap B_X(\widetilde x,r)=\emptyset$ for every $\widetilde x \in L_0\setminus\{x\}$. Now, the open set $V=\bigcup_{\widetilde  x\in L_0\setminus\{ x\}} B_X(\widetilde  x,r)$ satisfies $L\subset V \cup B_X( x,r)$ and $V \cap B_X(x,r)=\emptyset$, which is impossible since $L$ connected.

Now, take $x_{j_1} \in L_0$ and let $M_1=\{ x\in L_0 \colon d(x_{j_1},x)<2r\}$. Note that $M_1$ have at least 2 points of $L_0$: $x_{j_1}$ and the one given by the previous claim. If $L_0\subset B_X(x_{j_1}, 4r)$, then $L\subset B_X(x_{j_1},5r)$ and $5r\geq \E_1(L)\geq e_n(L)$. Thus
$$
\dfrac{e_{n+1}(L)}{e_{n}(L)}\geq \dfrac{e_{n+1}(L)}{5r}=\dfrac{e_{n+1}(L)}{5e_{n+1}(L)(1+\delta)}=\dfrac{1}{5(1+\delta)}
$$ and we are done, since $\delta$ is arbitrary.
If, on the contrary, there exists $x_{j_2} \in L_0$ such that $d(x_{j_2},x_{j_1})\geq 4r$, we take $M_2=\{x\in L_0 \colon d(x_{j_2},x)<2r\}$. By the claim above, $M_2$ has at least 2 points and, clearly,  $M_2\cap M_1=\emptyset$. Now, if $L_0\subset B_X(x_{j_1},4r) \cup B_X(x_{j_2},4r)$, then $L\subset B_X(x_{j_1},5r) \cup B_X(x_{j_2},5r)$. This implies that $5r\geq\E_2(L)\geq e_n(L)$, which gives  $$\dfrac{e_{n+1}(L)}{e_{n}(L)}>\dfrac{1}{5(1+\delta)} $$ and we are done.  If not, there exists $x_{j_3} \in L_0$ such that $x_{j_3}\notin B_X(x_{j_1},4r) \cup B_X(x_{j_2},4r)$ and we take $M_3=\{x \in L_0 \colon d(x_{j_3},x)<2r\}.$ Again by the claim, $M_3$ have at least 2 points and, also, $M_1\cap M_3=\emptyset$ and $M_2\cap M_3 =\emptyset$. Then, if $L_0\subset \bigcup_{i=1,2,3} B_X(x_{j_i},4r)$, reasoning as before we obtain that $$\dfrac{e_{n+1}(L)}{e_{n}(L)}\geq \dfrac{1}{5(1+\delta)}.$$ If not, we continue with this procedure and, since $L_0$ have $2^n$ points, this procedure comes to an end. So for some $m\in \mathbb N$  we get subsets $M_i\subset L_0 $ and points  $x_{j_i} \in M_i$,  $1\leq i \leq m $,  such that $L_0\subset \bigcup_{i=1}^{m} B_X(x_{j_i},5r)$, which implies that that $5r\geq \E_m(L)$.  Since the sets $M_i$ are disjoint and each has at least 2 points, then $m\leq \frac{2^n}{2}=2^{n-1}$. Thus {$5r\geq \E_{m}(L)\geq e_n(L)$ and then $$\dfrac{e_{n+1}(L)}{e_{n}(L)}\geq\dfrac{1}{5(1+\delta)},$$}
which give the desired result.
\end{proof}

We finish this section with an example that will be used later. This example shows  a connected set $K$ in a Banach space $E$ for which the sequence $(e_n(K))_n$ belongs to $\ell_1$ while  $\dim_BK=\infty$ which, thanks to the above proposition, is equivalent to $\limsup_{n\to \infty}e_n(K)^{1/n}=1$.

\begin{example}\label{ejemplo}
For $0<\varepsilon<1$, consider the set $K=\{(x_n)_n \subset \mathbb C\colon |x_n|\leq\varepsilon^n\}\subset c_0$. Then $\diminf_{B}K=\infty$ and  $(e_n(K))_n \in \ell_1$.
\end{example}

\begin{proof}
Fix $N \in \mathbb N$ and denote by $\Pi_N\colon c_0\rightarrow c_0$ the projection onto the first coordinates. Let $K^N=\Pi_N(K)$. Note that $K^N\subset K$ and that for any $(x_n)_n\in K$,  $\|\Pi_N ((x_n)_n) -(x_n)_n\|\leq \ep^N$. Then we have the inequalities
\begin{equation}\label{eq1}
e_n(K^N)\leq e_n(K)\leq e_n(K^N)+\ep^N \quad \mbox{for all } n \in \mathbb N\quad
\end{equation}

We define the diagonal operator $D_N\colon c_0\rightarrow c_0$  by $D((x_n)_n)=(x_1 \ep, x_2 \ep^2,\ldots, x_N \ep^N,0,\ldots)$.  Since $D_N(B_{c_0})=K^N$, we can apply \cite[Proposition~1.3.2]{CaSt} to estimate the entropy numbers of $K^N$ as
\begin{equation}\label{eq12}
	\sup_{1\leq k \leq N} 2^{-\frac{(n-1)}{2k}} \ep^{\frac{k+1}{2}}\leq e_n(K^N)\leq 6 \sup_{1\leq k < N} 2^{-\frac{(n-1)}{2k}} \ep^{\frac{k+1}{2}}.
\end{equation}
Combining \eqref{eq1} and \eqref{eq12}, we obtain
$$
\sup_{1\leq k \leq N} 2^{-\frac{(n-1)}{2k}} \ep^{\frac{k+1}{2}}\leq e_n(K)\leq 6 \sup_{1\leq k < N} 2^{-\frac{(n-1)}{2k}} \ep^{\frac{k+1}{2}} +\ep^N.
$$
The above inequality holds for every $N\in \mathbb N$, so writing $s=\min\{\varepsilon, \frac{1}{2}\}$ and $S=\max\{\varepsilon, \frac{1}{2}\}$, and using simple calculations, we find positive constants $C_1$ and $C_2$ such that
$$
C_1 s^{\sqrt{n-1}}\leq e_n(D)\leq C_2 S^{\sqrt{n-1}}.
$$
Thus, $(e_n(D))_n \in \ell_1$.

To see that $\dim_BK=\infty$,  fix again $N\in \mathbb N$ and define $T_N\colon c_0\rightarrow \mathbb C^N$ by $T_N((x_n)_n)=(x_1,x_2,\ldots,x_N)$. Since $T_N(K)\subset \mathbb C^N $ has non-empty interior, we obtain that
$2N=\diminf_B(T_N(K))\leq \diminf_B(K)$. Since $N$ was arbitrary, the result follows.
\end{proof}


\section{On the dimension of the image}\label{Sec: Dimension y entropy numbers}

We begin this section with the following simple observation: if  $E$ and $F$  are normed spaces,   for a linear operator $T\colon E\to F$ there is clear relationship between the box dimension of $T(B_E)$ and its linear dimension.  Indeed, if $\dim T(E)=N$ is finite, then $T(E)$ is an isomorphic copy of  $\mathbb C^N$ and $T(B_E)$ corresponds (via such isomporphisim) to an open subset of $\mathbb C^N$. Then,  we have $\dim_B(T(B_E)) =2N$ (which means, in particular, that $T(B_E)$ has finite box dimension). On the other hand, if $T(E)$ has infinite linear dimension, we can take $\{x_n\}_n$ a sequence of linearly independent elements in $T(B_E)$. Now, for each $N$, $\Span\{x_1,\dots,x_N\} \cap T(B_E)$ is homeomorphic (via the restriction of a linear isomorphism) to an open subset of $\mathbb C^N$. Since bi-Lipschitz mappings preserve box dimension, we have $$N=\dim_B(\Span\{x_1,\dots,x_N\} \cap T(B_E)) \le \dim_B ( T(B_E)))$$ for all $N$. Therefore, the box dimension of  $T(B_E)$ is also infinite. In particular, we have the following.
\begin{remark}\label{rem: lineal}
  If  $E$ and $F$  are normed spaces and  $T\colon E\to F$ is a linear operator, then $T(B_E)$ has infinite (linear) dimension if and only if it has infinite box dimension.
\end{remark}
The aim of this section is to study possible analogous results for homogeneous polynomials and holomorphic functions between normed spaces.  In the polynomial case, we obtain a result analogous to Remark \ref{rem: lineal}, but the proof is much more involved. Our result in this direction is the following.

\begin{theorem}\label{thm: BoxdimensionPoly}
Let $P\colon E\to F$ be a homogeneous polynomial. Then,  $\dim P(B_E)=\infty$ if and only if  $\diminf_B P(B_E)=\infty$.
\end{theorem}

We devote Section \ref{sec:pol} to the proof of this  theorem, which involves some results on polynomials of several complex variables that we think are interesting in their own.  Regarding holomorphic mappings, the situation is completely different: Example \ref{Example: holomorphic function} below shows that the box dimension of the image of $f$ may be finite while its linear dimension is infinite. Let $f\colon U\rightarrow F$ where $U\subset E$ is an open set. It is clear that, if $F$ is finite dimensional, then  $f(U)$ has finite both linear and box dimensions. If $E$ is finite dimensional, things do not go so smoothly (see again  Example \ref{Example: holomorphic function}). We start with the following simple positive result.

\begin{remark}\label{prop:domfindim} Let $E$ and $F$ be Banach spaces, $E$ finite dimensional, and $U\subset E$ be an open subset. If $f:U\to F$ is holomorphic and $K\subset U$ is compact, then $\dimsup_B f(K)$ is finite.
\end{remark}
\begin{proof}
By standard compactness arguments, it is enough to show the result for $K$ a closed ball (in the finite dimensional Banach space $E$). Since $f$ is holomorphic, it is continuously (Fr\'echet) differentiable. If $Df(z)$ denotes the differential of $f$ at $z$, let $M$ be the maximum of $\|Df(z)\|$ for $z\in K$. By \cite[Theorem 13.8]{Mu} we have
{
	$$\|f(x)-f(y)\|\leq \|x-y\|\cdot\sup_{0\leq t \leq 1} \|Df( x+t(y-x)\|\leq K \|x-y\|.$$}
Therefore, $f$ is Lipschitz on $K$ and then $\dimsup_B(f(K))\le \dimsup_B(K)<+\infty$.
\end{proof}

The following examples show, on the one hand, that an analogue to Theorem \ref{thm: BoxdimensionPoly} does not hold for holomorphic mappings. Also, we see in Example \ref{Example: holomorphic function}  that the restriction to a compact subset of $U$ is necessary in the previous proposition.

\begin{example}\label{Example: holomorphic function}
For $1\leq p \le \infty$, let $f\colon \Delta\rightarrow \ell_p$  ($c_0$ if $p=\infty$) be given by $f(z)=(z,z^2,z^3,\ldots)$. Then $f(\Delta)$ is not a relative compact set and, in particular, $\diminf_{B}f(\Delta)=\infty$.  Also, for any  $0<r<1$,  the linear dimension of  $ f(r \Delta)$ is infinite while its box dimension is finite. In other words, $f(r \Delta)$ has finite box dimensional but it is not contained in any finite dimensional subspace of $\ell_p$.
\end{example}
\begin{proof}
Take the sequence $(z_n)_n\subset \Delta$ given by $z_n=(\frac{1}{2})^{\frac{1}{2^{n-1}}}$. Note that, if $n<m$, then $$\|f(z_n)-f(z_m)\|_{\ell_p}\geq |\left((f(z_n)-f(z_m)\right)_{2^{m-1}}|=\big|\big(\frac{1}{2}\big)^{\frac{2^{m-1}}{2^{n-1}}} - \frac{1}{2}\big|\geq \big|\frac{1}{4} - \frac{1}{2}\big|=\frac{1}{2}.$$ This shows that  the sequence $(f(z_n))_n\subset f(\Delta)$ is  uniformly separated, and hence $f(\Delta)$ cannot be totally bounded.

By Remark \ref{prop:domfindim} we know that for $0<r<1$ we have $\dimsup_B(f(r\Delta))<\infty$. To see that  $f(r \Delta)$  has infinite linear dimension, fix $\delta< r$ and define $w_n=\frac{\delta}{n}$ for $n\in \mathbb N$. Let us see that $\{(f(w_n))_n\}$ is a linearly independent set. If not, for some $N\in \mathbb N$ and some not all zero scalars $a_1,\ldots, a_N$ we must have $\sum_{n=1}^N a_n f(z_n)=0,$ and so
$$
\left(\begin{array}{llll}
\delta & \ \, \frac{\delta}2 & \cdots & \ \, \frac{\delta}{N}\\
\delta^2 & \left(\frac{\delta}{2}\right)^2& \cdots & \left(\frac{\delta}{N}\right)^2\\
\  \vdots&\ \ \vdots&\ \, \vdots&\ \ \vdots\\
\delta^{N} & \left(\frac{\delta}{2}\right)^{N}& \cdots & \left(\frac{\delta}{N}\right)^{N}
\end{array}\right) \left(\begin{array}{c} a_1\\
a_2\\
\vdots\\
a_N\end{array}\right) = \left(\begin{array}{c} 0\\
0\\
\vdots\\
0\end{array}\right).
$$
This is impossible since the (Vandermonde) determinant of the matrix is nonzero.
\end{proof}

With the same ideas we can prove the following.

\begin{example}\label{Example: extension} Consider entire function $f:\mathbb C\to \ell_p$ given by $f(z)=(z,\frac{z^2}{2!},\frac{z^3}{3!},\ldots)$. Then, for every $r>0$ the linear dimension of  $ f(r \Delta)$ is infinite while its box dimension is finite.
\end{example}


\section{The proof of Theorem \ref{thm: BoxdimensionPoly}}\label{sec:pol}

Let $P:E\to F$ be a $m$-homogenoeous polynomial.
If $\dim\ P(B_E)<\infty$, then it is clear that  $\diminf_B P(B_E)<\infty$, since it is a subset of a finite dimensional vector space.

Suppose now that   $\dim\ P(B_E)=\infty$. Given $N\in \mathbb N$ we can take $x_1,\ldots,x_N\in B_E$ such that $\{P(x_1),\ldots,P(x_N)\}$ are lineary independent.
We choose a bounded linear projection $R\colon F\rightarrow \Span\{P(x_1),\ldots,{P(x_N)}\}$.  The multinomial formula gives us
\begin{equation}\label{eq2}
\begin{array}{rl}
{R\circ P\left(\sum_{i=1}^Na_ix_i\right)}=&R\Big(\displaystyle \sum_{k_1+\ldots+k_N=n} \dfrac{n!}{k_1\ldots k_N!} \prod_{1\leq j \leq N} a_j^{k_j} \overset \vee P(x_1^{k_1},\ldots,x_N^{k_N})\Big)\\
=&p_1(a_1,\ldots,a_n)P(x_1)+\ldots+p_N(a_1,\ldots,a_N)P(x_N)
\end{array}
\end{equation}
where, for each $1\leq j\leq N$, $p_j:\mathbb C^N\to \mathbb C$ is an $m$-homogeneous polynomial. Note that $p_j(\epsilon_j)=1$ and $p_j(\epsilon_i)=0$ if $i\neq j$. This means that our polynomials $p_1,\dots,p_N$ are linearly independent.
We define $\bff: \mathbb C^N \to \mathbb C^N$  by $$\bff(z_1,\ldots,z_N)=(p_1(z_1,\ldots,z_N),\ldots, p_N(z_1,\ldots,z_N))$$ and  $T\colon \mathbb C^N\rightarrow \Span\{P(x_1),\ldots P(x_N)\}$  by $$T(z_1,\ldots,z_N)=\sum_{j=1}^Nz_j P(x_j),$$
which is a linear isomorphism.

Since  $ \{P(a_1x_1+\ldots+a_Nx_N)\colon \sum_{j=1}^{N}|a_n|\leq 1\} \subset P(B_E),$ we have

$$
\dim_B(P(B_E))\geq \dim_B\left\{R\circ P(a_1x_1+\ldots+a_Nx_N)\colon \sum_{j=1}^{N}|a_j|\leq 1\right\}=\dim_B(\bff(B_{\ell_1^N})).
$$
The proof of Theorem \ref{thm: BoxdimensionPoly} is finished once we show that $\dim_B(\bff(B_{\ell_1^N}))$ goes to infinity as $N$ does. This is shown in Corollary \ref{coro:dimension} below,  where we see that $\dim_B(\bff(B_{\ell_1^N}))$ is at least $N^{\frac{1}{m}}-m$. In order to get Corollary \ref{coro:dimension} we need some preparation, which we develop in the following subsection.


\subsection{On the rank of the jacobian of homogeneous polynomials}\label{subsec}

Consider $p_1,\ldots, p_r$ $m$-homogeneous polynomials from $\mathbb{C}^N$ to $\mathbb{C}$ such that the $r\times N$ jacobian matrix $J_{p_1,\ldots, p_N}:=\left(\frac{\partial p_i}{\partial z_j}\right)_{1\leq i\leq r, 1\leq j\leq N}$ has maximal rank. Assume w.l.o.g. that the first principal minor $$Q\colon=Q(z_1,\ldots, z_N)=\det\left(\frac{\partial p_i}{\partial z_j}(z_1,\ldots,z_N)\right)_{1\leq i,j\leq r}$$
is not zero. For $i=1,\ldots, r,\, j=r+1,\ldots, N$, denote with $Q_{ij}=Q_{ij}(z_1,\ldots, z_N)$ the determinant of the submatrix of the jacobian matrix whose columns are indexed by $\{1,\ldots, r\}\cup\{j\}\setminus\{i\},$ multiplied by $(-1)^i.$
All these polynomials $Q$ and $Q_{ij}$ are homogeneous of degree $r(m-1)$.

Is easy to check that any other $m$-homogeneous polynomial $P$ such that the rank of the jacobian matrix of $p_1,\ldots, p_r, P$ is equal to $r$ must be a solution of the following linear system of partial  differential equations:
 \begin{equation}\label{ozu}
 Q\,\frac{\partial P}{\partial z_j}+\sum_{i=1}^r Q_{ij}\,\frac{\partial P}{\partial z_i}=0, \   j=r+1,\ldots, N.
\end{equation}
We will  exhibit a bound on the dimension of the $\C$-vector space of homogeneous polynomials of degree $m$ satisfying \eqref{ozu} which does not depend on $N$.
\begin{theorem}\label{mt}
The dimension of the $\C$-vector space of all $m$-homogeneous polynomials from $\mathbb C^N$ to $\mathbb C$  satisfying \eqref{ozu} is bounded from above by ${r+m-2 \choose m-2}$.
\end{theorem}

\begin{proof}
As $Q\neq0,$ after a linear change of variables in $\mathbb C^N$ if necessary, we may assume that the monomial $z_1^{r(m-1)}$ appears in the Taylor expansion of $Q$. For $\alpha=(\alpha_1,\ldots, \alpha_N)\in(\Z_{\geq0})^N,$ we denote with $z^\alpha$ the product $z_1^{\alpha_1}\ldots z_N^{\alpha_N},$ and we set $|\alpha|=\alpha_1+\cdots +\alpha_N.$
We can then write
$$\begin{array}{ccl}
Q&=& \sum_{|\alpha|=r(m-1)}Q_\alpha z^\alpha\\
Q_{ij}&=& \sum_{|\beta|=r(m-1)}Q_{ij\beta} z^\beta\\
P&=&\sum_{|\gamma|=m} P_\gamma z^\gamma.
\end{array}
$$
 Let $\epsilon_1,\ldots, \epsilon_n$ be the elements of the standard basis of $\Z^n.$ By setting to zero all the coefficients in the polynomial  \eqref{ozu}, we get the following linear system of equations for $(P_\gamma)_{|\gamma|=m}:$
 \begin{equation}\label{oxis}
 \sum_{\alpha+\gamma=\delta+\epsilon_j} \gamma_j\,Q_{\alpha}\,P_\gamma +\sum_{i=1}^r \sum_{\beta+\gamma=\delta+\epsilon_i} \gamma_i\,Q_{ij\beta}\,P_\gamma=0;  \ \ \ j=r+1,\ldots, n,\ |\delta|=r(m-1)+m-1.
 \end{equation}
 This is a homogeneous system of  $n+r(m-1)+m-2\choose r(m-1)+m-1$ linear equations in the ${n+m-1\choose m}$ variables $(P_\gamma)_{|\gamma|=m}.$ We will show that the corank of this system is bounded from above by the number of monomials of degree $m$ in $r$ variables, which is equal to ${r+m-1\choose m-1}$.

 To do so, we will show that  for any $\tilde\gamma\in(\Z_{\geq0})^N$ such that $|\tilde\gamma|=m,$ and $\tilde\gamma_j>0$ for some $j>r,$ there is an equation from \eqref{oxis} from where we can express $P_{\tilde\gamma}$ as a function of all those $P_{\gamma}$ with $z^{\gamma}\prec z^{\tilde\gamma}$ in the standard lexicographic order $z_1\prec x_2\prec\ldots\prec z_N.$ This will amount to a triangular submatrix of the linear system \eqref{oxis} with the desired corank.

 Write then $\tilde\gamma=(\tilde\gamma_1,\ldots, \tilde\gamma_j,0,\ldots,0),$ with $\tilde\gamma_j\neq0,\, j>r,$ and set
 $$\tilde\delta:=r(m-1)+\tilde\gamma_1,\tilde\gamma_2,\ldots, \tilde\gamma_{j-1},\tilde\gamma_j-1,0,\ldots,0).$$
 We extract the equation corresponding to $\delta=\tilde\delta$ in \eqref{oxis}, and find that
 \begin{itemize}
 \item $P_{\tilde\gamma}$ appears in a nonzero term in the first sum (corresponding to $\tilde\alpha=(r(m-1),0,\ldots, 0),$ recall that by hypothesis we have that $Q_{\tilde\alpha}\neq0$ and also $\tilde\gamma_j\neq0);$
 \item all the $P_\gamma$ appearing in the second (double) sum have $\gamma_\ell=0$ if $\ell>j$ and $\gamma_j\leq\tilde\gamma_j-1,$ so we have that $z^\gamma\prec z^{\tilde\gamma};$
 \item the rest of the $P_\gamma$ appearing in the first sum must satisfy
 $$\begin{array}{ccl}
 \gamma_1&\leq &r(m-1)+\tilde\gamma_1\\
 \gamma_2&\leq& \tilde\gamma_2\\
 \vdots&&\vdots\\
 \gamma_j&\leq &\tilde\gamma_j\\
 \gamma_\ell&=&0 \ \forall \ell>j.
 \end{array}
 $$
 As $\gamma\neq\tilde\gamma$ then  there exists a unique $i>1$ such that $\gamma_i<\tilde\gamma_i$ and $\gamma_\ell=\tilde\gamma_\ell$ for $\ell>i.$ This implies that
 $z^\gamma\prec z^{\tilde\gamma}$ also in this case,  which concludes with the proof of the claim.
 \end{itemize}
 \end{proof}

 \begin{remark}
 The bound is sharp as it can be seen easily by choosing as $P_i=z_i^m$ for $i=1,\ldots, r.$ Then, the $\C$-vector subspace of all the polynomials of degree $m$ in the variables $z_1,\ldots, z_r$ satisfy \eqref{ozu}.  On the other hand, it is a classical result (see \cite{sylv53}) that one can make a linear change of variables such that the polynomial system depends polynomially on  $r<N$ variables if and only if the equations \eqref{ozu} can be defined with $Q,\,Q_{ij}\in\C.$  \end{remark}

 \begin{corollary}\label{coro:dimension}
 Let $p_1,\ldots, p_N$ $m$-homogeneous polynomials from $\mathbb C^N$ to $\mathbb C$ which are linearly independent. Then the rank of the jacobian matrix of this family (i.e. the dimension of the image of the  map ${\bff}:\C^N \rightarrow\C^N$ defined by these polynomials)  is at least $N^{\frac1m}-m.$
 \end{corollary}
 \begin{proof}
 Denote with $r$ the rank of $J_{p_1,\ldots, p_N}$, and suppose w.l.o.g. that $J_{p_1,\ldots, p_N}$ has maximal rank.  Then, we have that all $p_j$ satisfies \eqref{ozu}  for $j=1,\ldots, N.$ As they are linearly independent,  from Theorem \ref{mt}, we deduce straightforwardly that
 $$ N\leq {r+m-1\choose m-1}\leq (r+m)^m. \qedhere
 $$
 \end{proof}


\section{Entropy numbers for holomorphic functions and their Taylor coefficients}\label{Sec: Holo vs pol}

In this section we relate the entropy numbers of $f$ with the entropy numbers of the polynomials of the Taylor series expansion of $f$ and viceversa. We start with a simple example. Let $f\colon B_{c_0}\rightarrow c_0$ be defined as $f((x_n)_n)=(x_1,x_2^2,x_3^3,\ldots)$. It is clear that $P_mf(0)((x_n)_n)=(0,\ldots,x^{m}_m,0,\ldots)$, getting that $\dimsup_{B}P_mf(0)(B_E)=2$. On the other hand,  given $\ep>0$ we have, {for every $n \in \mathbb N$}

$${e_n(}f(\ep B_{c_0}){)=e_n(}\{(x_n)_n \subset \mathbb C\colon |x_n|\leq \ep^n\}{)},$$ which are the entropy numbers of the set from Example~\ref{ejemplo}. Thus, we see that  $\dimsup_{B} f(\ep B_{c_0})=\infty$ for every $\ep>0$. This example shows the (expected) fact that the homogeneous polynomials in the Taylor expansion of $f$ can have small dimensional images while $f$ maps any ball in an infinite dimensional set. In particular, by  Proposition \ref{prop: dimension} (and its proof) we have that
 $$\limsup_{n\rightarrow +\infty} e_n (P_mf(0)(B_{c_0}))^{1/n} \le \frac{\sqrt{2}}{2} \text{ for all }m$$
 while, on the other hand, $$ \qquad \limsup_{n\rightarrow +\infty} e_n (f(\ep B_{c_0}))^{1/n}=1.$$ This  shows that a (uniformly) fast decay of the entropy numbers of $P_mf(0)(B_E)$ does not imply a fast decay in the entropy numbers of $f(\ep B_E)$ for $\ep>0$.

Recall that, by  Example~\ref{ejemplo},  the sequence $(e_n(f(\ep B_{c_0})))_n $ belongs to $\ell_1$  for every $\ep>0$. The following example shows a bit more.

\begin{example}
  There is an holomorphic function $f\colon B_{c_0}\rightarrow c_0$ such that every polynomial of its Taylor expansion at $0$ has finite rank (hence, for every $\ep>0$, $\dim_B P_mf(0) (\ep B_{c_0})$ is finite) but  for any  $\ep>0$ the sequence $(e_n(f(\ep B_{c_0}))_n$ does not belong to $\ell_p$ for any  $1\leq p<\infty$.
\end{example}
\begin{proof}

Consider $(\sigma_m)_m$ the partition of the natural numbers such that each $\sigma_m$ is a finite set with $m!$ consecutive elements:
$$
\sigma_1=\{1\}; \quad \sigma_2=\{\underbrace{2,3}_{2!}\}; \sigma_3=\{\underbrace{4,5,6,7,8,9}_{3!}\}; \quad \sigma_4=\{\underbrace{\ldots}_{4!}\}; \quad \ldots
$$
and define $f\colon B_{c_0}\rightarrow c_0$ by $$f((x_n)_n)=(x_1,\underbrace{x_2^2,x_3^2}_{2!}, \underbrace{x_4^3,x_5^3,\ldots,x_9^3}_{3!},\ldots).$$  In other words, the $j$-th coordinate of $f((x_n)_n)$ is $x_j^N$ if $j\in \sigma_N$.
Note that for every $m\in \mathbb N$ we have  $$P_mf(0)((x_n)_n)=(0,0,\ldots,\underbrace{x_j^m,x_{j+1}^m,\ldots,}_{m!},0\ldots),$$ where $j, j+1,\ldots \in \sigma_m.$ Now, denote by $\Pi_{\sigma_m}\colon c_0\rightarrow c_0$ the projection onto the coordinates  belonging to $\sigma_m$.
To see that $ (e_n(f(\ep B_{c_0})))_n$ does not belong to $\ell_p$ for any $\varepsilon >0$, we may suppose that $\ep=2^{-r}$ for some $r\in \mathbb N$.
Note that a sequence $(x_n)_n $ belongs to $ f(\Pi_{N}(2^{-r}B_{c_0}))$ if and only if $|x_j|\leq 2^{-rN}$ for $j \in \sigma_N$ and $x_j=0$ otherwise. Then, applying for instance \cite[1.3.2]{CaSt}, we have
$$
e_{N!+1}\left(\Pi_{\sigma_N} f(2^{-r}B_{c_0})\right)=\sup_{1\leq k \leq N!} 2^{- \frac{N!}{2k}} 2^{-rN}=\frac{2}{\sqrt 2} \ 2^{-rN},
$$
Now, for $1\leq p<\infty$ we  have
$$
\begin{array}{rl}
\dsum_{n=1}^{\infty} e_n(f(2^{-r}B_{c_0}))^p&=e_1(f(2^{-r}B_{c_0}))^p+ \dsum_{N=1}^{\infty} \dsum_{n=N!+1}^{(N+1)!} e_n(f(2^{-r}B_{c_0}))^p\\
&\geq  \dsum_{N=2}^{\infty} \dsum_{n=N!+1}^{(N+1)!} e_{N!+1}(f(2^{-r}B_{c_0}))^p\\
&\geq  \dsum_{N=2}^{\infty} \dsum_{n=N!+1}^{(N+1)!}e_{N!+1}(\Pi_{\sigma_N}(f(2^{-r}B_{c_0})))^p\\
&\geq\dsum_{N=2}^{\infty} ((N+1)!-N!) 2^{-p(\frac{1}{2} + rN)} \\
&\geq \dsum_{N=2}^{\infty} 2^{-\frac{1}{2}} N! N(2^{-r p})^{N}.
\end{array}
$$
Since the last term diverges for every $r$ and $p$, we are done. 
\end{proof}
Note that in the above example, $f(\Pi_{N}(2^{-r}B_{c_0}))$ coincides with $P_Nf(0)(2^{-r}B_{c_0})$ and, also, that the dimension of  $ P_mf(0)(2^{-r}B_{c_0})$ is finite for every $m$ but goes to infinity as $m$ does. It would be interesting to know if the sequence $(e_n(f(\ep B_E)))_n$ must decay fast if the dimensions of $ P_mf(0)(2^{-r}B_{c_0})$ are uniformly bounded.

\medskip 

The last goal of  this note is to find some bounds for the entropy numbers of $P_mf(x_0)(\ep B_E)$ in terms of properties of $f$.
Let us start by taking  a  standard  approach. Following  (the proof of) \cite[Proposition~3.4]{ArSc}, for an holomorphic mapping $f\colon U\rightarrow F$ and $x_0\in U$, there is $\ep>0$ such that \begin{equation}\label{eq:AS}
                                         P_mf(x_0)(\ep B_E)\subset \overline{\coe(f(x_0+\ep B_E))},
                                       \end{equation} where $\coe$ denotes the absolutely convex hull of a set.  Thus,  a good natural starting point to estimate the entropy numbers of $P_mf(x_0)(\ep B_E)$ in terms of those of $f(x_0+\ep B_E)$ is to estimate
$e_n(\coe(f(x_0+\ep B_E))$ in terms of $e_n(f(x_0+\ep B_E))$. Now, suppose that $\dimsup_B f(x_0+\ep B_E)=N<\infty$. By \cite[Corollary~5]{Foll}, we get that for every $\ep>0$, $$\E_n (f(x_0+\ep B_E))\leq 2(n+1)^{-1/(N+\ep)}.$$ From this estimate and   \cite[Proposition~4.5]{CKP},  we get \begin{equation}\label{eq:follo}
  e_n(\coe(f(x_0+\ep B_E)))\leq C (n+1)^{-1/N}.
\end{equation} for some $C>0$ (note that, when we  take $\coe$, we \emph{pay} the price of changing $\E_n$ to $e_n$). Finally,  \eqref{eq:AS} and \eqref{eq:follo} give the following estimate for the entropy numbers of $P_mf(x_0)(\ep B_E)$.
$$
e_n(P_mf(x_0)(\ep B_E))\leq e_n(\coe(f(x_0+\ep B_E)))\leq C (n+1)^{-1/N}.
$$
This bound allows us to deduce for example, that $(e_n(P_mf(x_0)(\ep B_E)))_n \in \ell_p$ for every $p>N$.
Our last theorem improves this claim, showing that under the same assumptions we actually have $(e_n(P_mf(x_0)(\ep B_E)))_n \in \ell_p$ for every $p>1$.

\begin{theorem}\label{prop: entropy en ellp} Let $E$ and $F$ be Banach spaces, $U\subset E$ be an open set, $x_0\in U$ and  $\ep>0$ be such that $x_0+\ep B_E \subset U$. Let $f\colon U\to F$ be a holomorphic function such that $\dimsup_B f(x_0+\ep B_E)<\infty$. Then, $(e_n(P_mf(x_0)(B_E))_{n\in \mathbb N} \in \ell_p$ for every $p>1$ and every $m \in \mathbb N$.
\end{theorem}

Before proving the theorem, we need a technical lemma. In what follows, for $a\in \mathbb R$, we write $\lceil a\rceil=\min\{k\in\mathbb{Z}| k\geq a\}$.

\begin{lemma}\label{Lemma: p.prin} Let $E$ and $F$ be Banach spaces, $U\subset E$ be an open set, $x_0\in U$ and  $\ep>0$ be such that $x_0+\ep B_E \subset U$. For an holomorphic function $f\colon U\to F$ and $n,m \in \mathbb N$, the following inequality holds
$$
e_{(n-1)C_n+1}(P_mf(x_0)(\ep B_E))\leq 2 e_n(f(x_0+\ep B_E))
$$
where $C_n=\left\lceil\dfrac{C}{e_n(f(x_0 +\ep B_E))}\right\rceil$ for some positive constant $C=C(f)$ which depends only on $f$.
\end{lemma}
\begin{proof}
	We may suppose that $x_0=0$ and $\ep=1$. Fix $n, m\in \mathbb N$. Given $\delta>0,$ there exists $M=\{y_1,y_2,\ldots,y_{2^{n-1}}\}\subset Y$ such that $f(B_E)\subset \bigcup_{k=1}^{2^{n-1}} y_k +\left(e_n(f(B_E))+\delta\right) B_F$.
	
Write $C=\sup_{x \in B_E} \|P_1f(x)\|$ and let $C_n=\left\lceil\dfrac{2\pi C}{e_n(f(B_E))}\right\rceil$. We split the interval $[0,2\pi]$ into $C_n$ disjoint intervals $J_1,\ldots J_{C_n}$ of length $\frac{2\pi}{C_n}$. Note that, for $x\in B_E$, and $t_0, t_1$ in one of this intervals, we have
$$
\|f(e^{it_0}x)-f(e^{it_1}x)\|\leq C |e^{it_0}-e^{it_1}|\leq C \frac{2\pi}{C_n}\leq e_n(f(B_E)).
$$

As a consequence, if  some  $t_0$ and some $y_k \in M$ satisfy $\|y_k - f(e^{it_0}x)\|\leq e_n (f(B_E))$ then, for  any other  $t$ in the same interval as $t_0,$ we have $\|y_k -f(e^{it}x)\|\leq 2 e_n (f(B_E))$. We define the set
$$L=\{y \in F\colon y=\dfrac{1}{2\pi} \displaystyle \sum_{j=1}^{C_n} \displaystyle \int_{J_j} z_j e^{-itm} dt,\text{ for some } z_1,\dots,z_{C_n} \in M\}.$$ Note that $L$ has $\left(2^{n-1}\right)^{C_n}=2^{C_n(n-1)}$ elements. The proof is complete if we show that for  $x \in B_E$, there exists $y \in L$ such that $\|y-P_mf(0)(x)\|\leq 2e_n(f(B_E))$. By the Cauchy integral formula (see for instance \cite[Corollary~7.3]{Mu}) we have
	$$
	P_mf(0)(x)=\frac{1}{2\pi} \int_0^{2\pi} f(e^{it}x) e^{-itm} dt=\frac{1}{2\pi} \displaystyle \sum_{j=1}^{C_n} \displaystyle \int_{J_j}f(e^{it}x) e^{-itm} dt.
	$$
	For each $j$, take $z_j \in M$ such that $\|z_j -f(e^{it}x)\|\leq 2 e_n(f(B_E))$ for all $t \in J_j$. Then, $y=\frac{1}{2\pi} \displaystyle \sum_{j=1}^{C_n} \displaystyle \int_{J_j} z_j e^{-itm} dt \in L$ and
	
		$$
		\begin{array}{rl}
			\|y-P_nf(0)(x)\|=&\displaystyle\left\|\frac{1}{2\pi} \sum_{j=1}^{C_n} \int_{J_j}\left(f(e^{it}x)-z_j\right) e^{-itm} dt\right\| \\
			\leq&\displaystyle\frac{1}{2\pi}\sum_{j=1}^{C_n}\int_{J_{j}}2 e_n (f(B_E)) dt\\
			= & 2e_n(f(B_E)).  \hfill \qedhere
		\end{array}
		$$
\end{proof}

\begin{proof}[Proof of Theorem \ref{prop: entropy en ellp}]
By  Lemma~\ref{Lemma: p.prin},  given $p>1$ we have

$$
\begin{array}{rl}
\displaystyle\sum_{n=1}^{\infty} e_n(P_mf(0)(B_E))^p =&\displaystyle \sum_{j=1}^{\infty}\sum_{n=(j-1)C_{j}+1}^{j \, C_{j+1}} e_n(P_mf(0)(B_E))^p\\
\leq &\displaystyle \sum_{n=1}^{\infty} (nC_{n+1} - (n-1)C_{n}) e_{(n-1)C_{n}+1}(P_mf(0)(B_E))^p\\
\leq &\displaystyle 2^p\sum_{n=1}^{\infty} (nC_{n+1} - (n-1)C_{n})e_{n}(f(B_E))^{p},
\end{array}
$$

We  finish the proof by showing that $$\displaystyle \limsup_{n\rightarrow \infty}\sqrt[n]{(nC_{n+1} - (n-1)C_{n})e_{n}(f(B_E))^{p}}<1.$$
Since $f(B_E)$ is a connected set, by Proposition~\ref{prop: liminf en} we have $$\displaystyle \liminf_{n\rightarrow \infty}\dfrac{e_{n+1}(f(B_E))}{e_n(f(B_E))}\geq\frac{1}{5}.$$ Thus, there exists   $n_0 \in \mathbb N$ such that for any $n>n_0$,
$$
\frac{n}{n^2+(n-1)C} \left(\frac{C}{\frac{e_{n+1}(f(B_E))}{e_n(f(B_E)}}+e_n(f(B_E))\right)<1.
$$
This implies that for every $n>n_0$
$$
n\left(\frac{C}{e_{n+1}(f(B_E))} +1\right) - \frac{(n-1)C}{e_n(f(B_E))}<\frac{n^2}{e_n(f(B_E))},$$
and since $\frac{C}{e_{n}(f(B_E))} \leq C_n\leq \frac{C}{e_{n}(f(B_E))}+1$, we get the inequality
$$
nC_{n+1}-(n-1)C_n \leq \frac{n^2}{e_n(f(B_E))}.
$$
From this last inequality we obtain that,
\begin{equation}\label{eq:4}
\sqrt[n]{(nC_{n+1}-(n-1)C_n)e_n(f(B_E))^p}\leq \left(\sqrt[n]{e_n(f(B_E))}\right)^{p-1} \sqrt[n]{n^2}.
\end{equation}

Since we are assuming that $\dimsup_B f(B_E)<\infty$, by Proposition~\ref{prop: dimension} we have that $\limsup_{n\rightarrow \infty} \sqrt[n]{e_n(f(B_e))}<1$. This and \eqref{eq:4} complete the proof.
\end{proof}

We do not know if Proposition~\ref{prop: entropy en ellp} holds for $p=1$. More precisely, we have the following question

\begin{questions}\label{q1} Let $E$ and $F$ be Banach spaces, $U\subset E$ an open set, and $x_0 \in U$. Take $f\colon U\rightarrow E$ an holomorphic function and suppose that there exists $\ep>0$ such that $\dim_B f(x_0+\ep B_E)<\infty$. Is it true that $(e_n(P_mf(x_0)(B_E))_n$ belongs to $ \ell_1$ for all  $m\in \mathbb N$?
\end{questions}

A probably more natural question, whose positive answer would clearly imply a positive answer to Question ~\ref{q1}, is the following. 

\begin{questions}\label{q2}
Let $E$ and $F$ be Banach spaces, $U\subset E$ an open set, and $x_0 \in U$.  Take $f\colon U\rightarrow E$ an holomorphic function and suppose that there exists $\ep>0$ such that $\dim_B f(x_0+\ep B_E)<\infty$. Is it true that for every $m\in \mathbb N$,  $\dimsup_B P_mf(x_0)(B_E)<\infty$?
\end{questions}

\bigskip

\noindent\textbf{Acknowledgments:} We thank Richard Aron whose questions motivated this work.  We also thank Pablo Gonz\'alez-Mazon, Gabriela Jeronimo and Daniel Perrucci for useful conversations regarding Theorem \ref{thm: BoxdimensionPoly}.

\end{document}